\documentclass[11pt]{article}
\usepackage{multicol}
\usepackage{indentfirst}
\usepackage{latexsym}
\usepackage{hyperref}
\usepackage{graphicx}
\usepackage{mathrsfs}
\usepackage{bookmark}
\usepackage{enumerate}
\usepackage{amsfonts,amsthm,amssymb,mathtools}
\usepackage{amsmath}
\usepackage{enumitem}
\usepackage{latexsym, euscript, epic, eepic, color}
\usepackage{amssymb}

\usepackage{enumerate}
\usepackage[all]{xy}
\usepackage{setspace}
\allowdisplaybreaks

\usepackage{epstopdf}
\numberwithin{equation}{section}
\newtheorem{theorem}{Theorem}[section]

\newtheorem{definition}[theorem]{Definition}

\allowdisplaybreaks

\begin{document}
\baselineskip=16pt

\title{ Hexagonal and \texorpdfstring{$k$}{}-hexagonal graph's normalized Laplacian spectrum and applications \footnote{Corresponding author: Hao Liu (lhmath@njtech.edu.cn)}
}

\author{ Hao Li$^{1,2}$, Xinyi Chen$^{1}$ and Hao Liu$^{1,*}$ \\
	%	\small $^{1}$  63763 army of PLA,\\
	%	\small  Lingshui, China, 572400.\\
    \small $^{1}$ School of Physical and Mathematical Sciences, Nanjing Tech University,\\
	\small  Nanjing, China, 211816.\\
	\small $^{2}$ College of Science, National University of Defense Technology, \\
	\small  Changsha, China, 410073.\\
	%\small $^{2}$  Sanya Station, China Xi’an Statellite Control Center,\\
	%\small  Sanya, China, 572427.\\
}

%\author{Yingui Pan\thanks{Corresponding author: Yingui Pan(panygui@163.com)},   Jianping Li   \\
%	\small  College of Liberal Arts and Sciences, National University of Defense Technology, \\
%	\small  Changsha, China, 410073.\\
%\small $^{b}$  School of Mathematics and Computational Science, Hunan First Normal University,\\
%\small  Changsha, China, 410205.\\
%\small $^{c}$  School of Mathematics and Statistics, Shandong Normal University,\\
%\small Jinan, Shandong, China,  250014.

\date{\today}

\maketitle

\begin{abstract}
 Substituting each edge of a simple connected graph $G$ by a path of length 1 and $k$ paths of length 5 generates the $k$-hexagonal graph $H^k(G)$.
        Iterative graph $H^k_n(G)$ is produced when the preceding  constructions are repeated $n$ times. According to the graph structure, we obtain a set of linear equations, and derive the entirely normalized  Laplacian spectrum of $H^k_n(G)$  when $k = 1$ and $k \geqslant 2$ respectively by analyzing the structure of the solutions of these linear equations. We find significant formulas to calculate   the Kemeny's constant, multiplicative degree-Kirchhoff index and number of spanning trees of $H^k_n(G)$ as applications.
\end{abstract}



\section{Introduction}
Spectral graph theory studies the properties of the graph by analyzing the eigenvalues and eigenvectors of some matrices such as adjacency, Laplacian and normalized Laplacian matrices, which can be calculated  according to the construction of the graph.
%Among all the matrices of the graph, the normalized Laplacian matrix are closely related to the random walk of the graph, and it has been widely studied.
%The $resistance\ distance$  $r_{ij}$ between two vertices $i$ and $j$ of a connected graph $G$ is equivalent to the resistance between the respective two points of an electrical network created to match to $G$ such that the resistance between any two adjacent points is unity \cite{resistance}. Similar to the standard distance, the resistance distance is intrinsic to the graph, not only with some fine purely mathematical properties, but also with a substantial potential for chemical applications \cite{Resistance_distance},
%Resistance distance has gradually become a hot research topicSince Klein and Randic \cite{Resistance_distance} proposed the concept of resistance distance in 1993. As an important topological index, resistance distance has a wide range of applications in random walks on graphs, harmonic function theory, network robustness analysis and chemical graph theory. 

In mathematical chemistry, the chemical structure can be represented by graphs in which vertices represent atoms and edges represent chemical bonds. This performance inherited a wealth of essential information regarding the chemical characteristics of molecules. Many physical and chemical properties of molecules are strongly related to the parameters of graph.
Multiplicative degree-Kirchhoff index  \cite{1} is one of these parameters. It is described by
$
    Kf^{'}\left( G \right) =\sum_{v<u}{d_vd_ur_{vu}},
$
in which $r_{vu}$  is the resistance distance  \cite{Resistance_distance} between vertex $v$ and vertex $u$.
$Kf^{'}\left( G \right)$ can be calculated by the normalized Laplacian spectrum of $G$  \cite{1}.
%Additional information regarding the normalized Laplacian and the multiplicative degree-Kirchhoff index can be found in . 
The normalized Laplacian matrix can also be used to calculate some other parameters of the graph, including  the number of spanning trees and Kemeny's constant.

%The $Kemeny's\ constant\ K(G)$ of $G$ is the number of steps needed to transit from the initial vertex to the expected target vertex, and the route is randomly selected according to the static distribution of unbiased random walks on $G$  \cite{kemeny}.
%Yang and Zhang \cite{h30} investigated the combinatorial Laplacian of the linear hexagonal chain, whereas
In this paper, we  consider a simple and connected graph $G$. Let $D(G)$ be the diagonal matrix of vertex degree of $G$ and $A(G)$ be the adjacency matrix of $G$. The Laplacian matrix of $G$ is $L(G)=D(G)-A(G)$. Let $
    P\left( G \right) =D\left( G \right) ^{-\frac{1}{2}}A\left( G \right) D\left( G \right) ^{-\frac{1}{2}}$,
then the $(i,j)$-entry of $P(G)$ is $\frac{\left( A\left( G \right) \right) _{ij}}{\sqrt{d_id_j}}$, where $d_i$ is the degree of vertex $i$. The normalized Laplacian matrix of $G$ is defined as
\begin{equation}\label{eq2}
    \mathcal{L} \left( G \right)
    =D\left( G \right) ^{-\frac{1}{2}}L\left( G \right) D\left( G \right) ^{-\frac{1}{2}}
    =I-P\left( G \right),
\end{equation}
in which $I$ is the unity matrix.

Eq.(\ref{eq2}) represents the bijective relationship between the spectra of $P\left( G \right)$ and $\mathcal{L} \left( G \right)$.
Let $X=(x_1,x_2,\cdots,x_N)^T$ represent an eigenvector corresponding to the eigenvalue $\lambda$ of $\mathcal{L} \left( G \right)$, where $N$ is the order of $G$. By Eq.(\ref{eq2}), we have $P\left( G \right)X=(1-\lambda)X$, which can be expressed by
\begin{equation}\label{eq3}
    \sum_{j=1}^N{\frac{\left( A\left( G \right) \right) _{ij}}{\sqrt{d_id_j}} x_j}=\left( 1-\lambda \right) x_i\,\,,\ i=1,2,\cdots N.
\end{equation}

Note that $\mathcal{L} \left( G \right)$ is hermitian  and similar to $D(G)^{-1}L(G)$, its eigenvalues are non-negative \cite{qua}. The normalized Laplacian spectrum of $G$ is defined by $S(G) =\left\{ \lambda _1,\lambda _2,\cdots ,\lambda _N \right\}$, in which $\lambda_i$, $i=1,2,\cdots,N$ are eigenvalues of $\mathcal{L} \left( G \right)$ such that $0=\lambda _1<\lambda _2\leqslant \cdots \leqslant \lambda _N \leqslant 2$, and when $G$ is bipartite, $\lambda _N = 2$  \cite{Spectralgraphtheory}.

According to  \cite{Hex_and_CHem}, benzenoid hydrocarbons are naturally represented by hexagonal systems. Therefore, hexagonal systems serve an essential function in mathematical chemistry. Numerous studies have been conducted on normalized Laplacian of hexagonal systems so far. For example,  Huang  \cite{h31} investigated the normalized Laplacian of the linear hexagonal chain. Li and sun offered the normalized Laplacian of the hexagonal Möbius chain and the zigzag polyhex nanotube   \cite{lisum}.

Given the importance of hexagonal graphs in chemistry  \cite{Hex_and_CHem}, and inspired by Li and Hou  \cite{qua}, Xie  \cite{tri}, and Huang  \cite{k-tri} et al.,  we obtain  the entirely normalized Laplacian spectrum of a special iterative graph of $G$ which turn each edge of $G$ into some hexagons. The significant formulas to calculate  the Kemeny's constants, multiplicative degree-Kirchhoff indices and numbers of spanning trees are also characterized as applications.

\section{The normalized Laplacian spectrum of \texorpdfstring{$H_{n}(G)$}{} and \texorpdfstring{$H^k_{n}(G)$}{}}

%In this section, we consider a simple connected graph $G(V(G),E(G))$ with vertex and edge set $V(G)$ and $E(G)$. 
%Let $N_0$ and $E_0$ denote the order  and the size of $G$.

\begin{definition}
    The $k$-hexagonal graph $H^k(G)$ of $G$, is obtained by substituting each edge of $G$ by a path of length 1 and $k$ paths of length 5. If $k = 1$, $H^1(G)$ is simplified to $H(G)$,  described as the hexagonal graph of $G$.
\end{definition}

Figure \ref{fig} provides an illustration of the hexagonal and $k$-hexagonal graphs.
\begin{figure}[ht]
    \centering
    \includegraphics[scale=0.2]{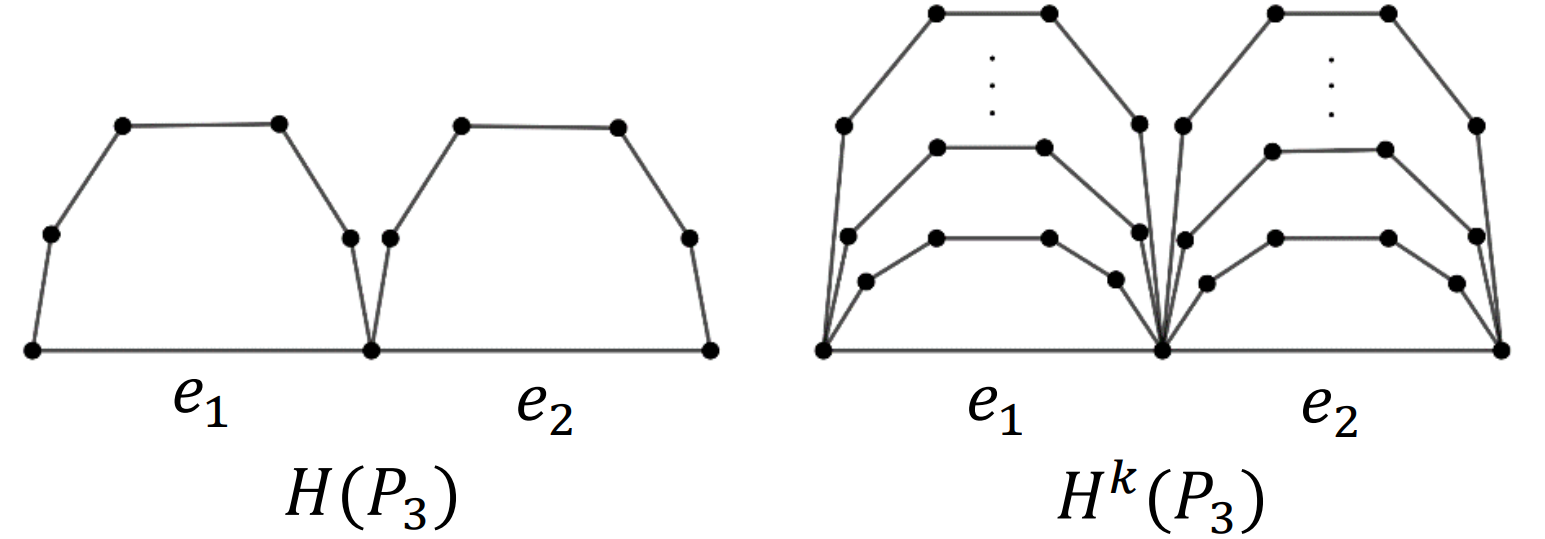}
    \caption{Graphs $H(P_3)$ and $H^k(P_3)$ ($P_3$ denotes a path of order 3).}
    \label{fig}
\end{figure}

Denote $H^k_0(G)=G$ and $H^k_1(G)=H^k(G)$. The $n$-th $k$-hexagonal graph is obtained through $H^k_{n}(G)=H^k(H^k_{n-1}(G))$. Let $N_n^k$ and $E_n^k$  be the order and size of $H^k_{n}(G)$, respectively. In particular, if $k=1$,  $N_n^k$ and $E_n^k$  are simplified to $N_n$ and $E_n$. Obviously, we have
$$
    N_{n}^{k}=N_{n-1}^{k}+4kE_{n-1}^{k},\ E_{n}^{k}=\left( 5k+1 \right) E_{n-1}^{k}
$$
and
$$
    N_{n}^{k}=N_{0}^{k}+\frac{4}{5}\left( \left( 5k+1 \right) ^n-1 \right) E_{0}^{k},\ E_{n}^{k}=\left( 5k+1 \right) ^nE_{0}^{k}.
$$
% In particular,we have
% $$
%     N_n=N_{n-1}+4E_{n-1},\ E_n=6E_{n-1}
% $$
% and
% $$
%     N_n=N_0+\frac{4}{5}\left( 6^n-1 \right) E_0,\ E_n=6^nE_0.
% $$

For convenience, when we refer to $H_n^k(G)$ in the following, we also mean $H_n^k(G),k\geqslant 2$.

As we will see, the characteristic of  $n$-th $k$-hexagonal graph  has a significant effect on the structures of the normalized Laplacian spectrum of $H_{n}(G)$ and $H^k_{n}(G)$. We demonstrate that this spectrum may be derived iteratively by the normalized Laplacian spectrum of  $G$.

\subsection{The normalized Laplacian spectrum of \texorpdfstring{$H_{n}(G)$}{}}
Let $m_{\mathcal{L}_{n}}(\sigma)$  represent the multiplicity of the eigenvalue $\sigma$ of matrix ${\mathcal{L}_{n}}$.
\begin{theorem} \label{theo1}
    %Let $G$ be a simple connected graph with $N_0$ vertices and $E_0$ edges, and $H_{n}(G)$ be the $n$-th hexagonal graph of $G$. 
    Denote by $\mathcal{L}_n $  the normalized Laplacian matrix of $H_{n}(G)$. The normalized Laplacian spectrum of $H_n(G)$, $n \geqslant 1$ is given below.

    \begin{enumerate}[label=(\roman*)]
        \item \label{theo1_1} If $\sigma $ is an eigenvalue of $ \mathcal{L}_{n-1} $  satisfying $\sigma  \ne 0,2$, then $\lambda_1,\lambda_2,\lambda_3$ are the eigenvalues of $ \mathcal{L}_{n} $
              with $m_{\mathcal{L}_{n}}(\lambda_1)=m_{\mathcal{L}_{n}}(\lambda_2)=m_{\mathcal{L}_{n}}(\lambda_3)=m_{\mathcal{L}_{n-1}}(\sigma)$, where $\lambda_1,\lambda_2,\lambda_3$ are the roots of
              \begin{equation}\label{eq18}
                  4\lambda ^3-\left( 10+2\sigma \right) \lambda ^2+\left( 5+4\sigma \right) \lambda -\sigma =0.
              \end{equation}

        \item \label{theo1_2} 0 is an eigenvalue of $\mathcal{L}_{n}$ with the multiplicity 1. While 2 is an eigenvalue of $\mathcal{L}_{n}$ with the multiplicity 1 if $G$ is bipartite;

        \item \label{theo1_3} $ \frac{1}{2}$ and $ \frac{3}{2}$  are the eigenvalues of $\mathcal{L}_n$ with multiplicity $N_{n-1}$.

        \item \label{theo1_4} $ \frac{5-\sqrt{5}}{4}$ and $ \frac{5+\sqrt{5}}{4}$  are the eigenvalues of $\mathcal{L}_n$ with multiplicity $E_{n-1}-N_{n-1}+1$.

        \item \label{theo1_5} When $G$ is bipartite,  $ \frac{3-\sqrt{5}}{4}$, $ \frac{3+\sqrt{5}}{4}$  are the eigenvalues of $\mathcal{L}_n$ with multiplicity $E_{n-1}-N_{n-1}+1$.

        \item \label{theo1_6} When $G$ is non-bipartite,  $ \frac{3-\sqrt{5}}{4}$, $ \frac{3+\sqrt{5}}{4}$  are the eigenvalues of $\mathcal{L}_n$ with multiplicity $E_{n-1}-N_{n-1}$.
    \end{enumerate}
\end{theorem} 

\begin{proof}\rm
    Partition the set of vertices in $H_n(G)$ as $V_{old} \cup V_{new}$, in which $V_{old}$ is the set of all the vertices inherited from $H_{n-1}(G)$, and $V_{new}$ is its complement.

    For any vertex $i\in V_{old}$, let $ V_0(i)\subset V_{old}$ be the set of old vertices in $H_n(G)$ adjacent to $i$,  and for each vertex $j_{0} \in V_{0}(i)$, $ij_{1}j_{2}j_{3}j_{4}j_0$ is the newly added parallel path of length 5 in $H_n(G)$ with respect to edge $ij_0$. Then we let $V_{1}\left( i \right) =\left\{ j_{1}:\forall j_0\in V_0(i) \right\}$ denote the set of second vertices $j_1$ in the path $ij_{1}j_{2}j_{3}j_{4}j_0$ corresponding to edge $ij_0$, for each vertex $j_{0} \in V_{0}(i)$. For convenience, we use $V_0$ to denote $V_0(i)$ and $V_1$ to denote $V_1(i)$. Similarly, let $V_{2}=\left\{ j_{2}:\forall j_0\in V_0 \right\},\ V_{3}=\left\{ j_{3}:\forall j_0\in V_0 \right\}$ and $V_{4}=\left\{ j_{4}:\forall j_0\in V_0 \right\}$.

    % Clearly, we have
    % $$
    % V_{new}=\bigcup_{i\in V_{old}}{\bigcup_{t=1}^4{V_t}}.
    % $$

    %For any vertex $i\in V_{old}$, let $ V_1\subset V_{new}$ denote the set of the new neighbors of vertex $i$ in $H_n(G)$ and $ V_0\subset V_{old}$ the set of its old neighbors. Then, let $V_2\subset V_{new}$ be the set of the neighbors of all vertices in $V_1$, $V_3\subset V_{new}$ be the set of the neighbors of all vertices in $V_2$ and $V_4\subset V_{new}$ be the set of the neighbors of all vertices in $V_3$. That is, $V_t$ is a subset of $V_{new}$ and all vertices in $V_t$ have paths of length $t$ to vertex $i$, where $t=1,2,3,4$. Obviously, for any vertex $j^{'}$ in $V_0$, there is a vertex $j^{''}$ in $V_4$ adjacent to it. There exists, from the definition of $ n-th\ hexagonal\ graph$, a bijection between $V_t$ and $V_0$. 

    Let $X=(x_1,x_2,\cdots,x_{N_n})^T$ denote an eigenvector associated with the eigenvalue $\lambda$ of $\mathcal{L}_n$ and $d_n(v)$ be the degree of vertex $v$ in $H_n(G)$, and  we have

    $$
        d_n\left( v \right) =\begin{cases}
            2d_{n-1}\left( v \right), \begin{matrix}
                                          & v\in V_{old}; \\
                                     \end{matrix}       \\
            2,\begin{matrix}
                  &  & \begin{matrix}
                           &  & \\
                      \end{matrix} & v\in V_{new} .\\
             \end{matrix} \\
        \end{cases}
    $$

    For any vertex $i\in V_{old}$, Eq.(\ref{eq3}) leads to

    \begin{equation}
        \begin{aligned}\label{pro1_1}
            \left( 1-\lambda \right) x_i
             & =\sum_{j_1\in V_1}{\frac{1}{2\sqrt{d_{n-1}\left( i \right)}}x_{j_1}}+\sum_{j_0\in V_0}{\frac{1}{2\sqrt{d_{n-1}\left( i \right) d_{n-1}\left( j_0 \right)}}x_{j_0}}.
        \end{aligned}
    \end{equation}

    Given vertex $j_0 \in V_{0}$, we consider the vertices of the parallel path corresponding to edge $ij_0$, which is denoted by $ij_{1}j_{2}j_{3}j_{4}j_0$.

    % For any vertex $j_0\in V_0$, a path of length 5 parallel to edge $i j_0$ is added between vertices $i$ and $j$, denoted as $R_i=i j_1 j_2 j_3 j_4 j_0$, where vertex $j_t\in V_t \subset V_{new},\ t=1,2,3,4$.

    For the vertex $j_1\in V_1$, we have a similar relationship
    \begin{equation}\label{pro1_2}
        \left( 1-\lambda \right) x_{j_1}=\frac{1}{2\sqrt{d_{n-1}\left( i \right)}}x_i+\frac{1}{2}x_{j_2}.
    \end{equation}

    Analogously, for the vertex $j_2\in V_2$, we obtain
    \begin{equation}\label{pro1_3}
        \left( 1-\lambda \right) x_{j_2}=\frac{1}{2}x_{j_1}+\frac{1}{2}x_{j_3}.
    \end{equation}

    For the vertex $j_3\in V_3$, we have
    \begin{equation}\label{pro1_4}
        \left( 1-\lambda \right) x_{j_3}=\frac{1}{2}x_{j_2}+\frac{1}{2}x_{j_4}.
    \end{equation}

    For the vertex $j_4\in V_4$, we get
    \begin{equation}\label{pro1_5}
        \left( 1-\lambda \right) x_{j_4}=\frac{1}{2}x_{j_3}+\frac{1}{2\sqrt{d_{n-1}\left( j_0 \right)}}x_{j_0}.
    \end{equation}

    Combining Eqs.(\ref{pro1_2}), (\ref{pro1_3}), (\ref{pro1_4}) and (\ref{pro1_5}), we get linear equations with $x_{j_1},x_{j_2},x_{j_3}, x_{j_4}$ as variables in matrix form, as follows,
    \begin{equation}\label{pro1_leqs}
        \left( \begin{matrix}
                2\left( 1-\lambda \right) & -1                        & 0                         & 0                         \\
                -1                        & 2\left( 1-\lambda \right) & -1                        & 0                         \\
                0                         & -1                        & 2\left( 1-\lambda \right) & -1                        \\
                0                         & 0                         & -1                        & 2\left( 1-\lambda \right) \\
            \end{matrix} \right) \left( \begin{array}{c}
                x_{j_1} \\
                x_{j_2} \\
                x_{j_3} \\
                x_{j_4} \\
            \end{array} \right) =\left( \begin{array}{c}
                \frac{1}{\sqrt{d_{n-1}\left( i \right)}}x_i       \\
                0                                                 \\
                0                                                 \\
                \frac{1}{\sqrt{d_{n-1}\left( j_0 \right)}}x_{j_0} \\
            \end{array} \right).
    \end{equation}
    The determinant of coefficient matrix of Equations (\ref{pro1_leqs}) is denoted by $Det(\lambda)$, that is
    $$
        Det\left( \lambda \right) =16\lambda ^4-64\lambda ^3+84\lambda ^2-40\lambda +5.
    $$

     \ref{theo1_1} When $\lambda \ne  \frac{2\pm \sqrt{2}}{2}, \frac{1}{2}, \frac{3}{2}, \frac{5\pm \sqrt{5}}{4},\frac{3\pm \sqrt{5}}{4}$, we have $Det(\lambda) \ne 0$ and $Det(\lambda)+1 \ne 0$. Equations (\ref{pro1_leqs}) has a unique solution, and $x_{j_1}$ can be determined by
    \begin{equation}\label{eq4}
        x_{j_1}=\frac{1}{Det\left( \lambda \right)}\left( \frac{1}{\sqrt{d_{n-1}\left( j_0 \right)}}x_{j_0}-\frac{8\lambda ^3-24\lambda ^2+20\lambda -4}{\sqrt{d_{n-1}\left( i \right)}}x_i \right) .
    \end{equation}

    Substituting Eq.(\ref{eq4}) into Eq.(\ref{pro1_1}) yields
    \begin{equation}
        \begin{aligned}\label{eq5}
            \left( 1-\lambda \right) x_i
            =\left( \frac{1}{Det\left( \lambda \right)}+1 \right) \sum_{j_0\in V_0}{\frac{1}{2\sqrt{d_{n-1}\left( i \right) d_{n-1}\left( j_0 \right)}}x_{j_0}}-\frac{4\lambda ^3-12\lambda ^2+10\lambda -2}{Det\left( \lambda \right)}x_i.
        \end{aligned}
    \end{equation}
    Since  $Det(\lambda)+1 \ne 0$, $ \frac{32\lambda ^5-160\lambda ^4+288\lambda ^3-224\lambda ^2+70\lambda -6}{-Det\left( \lambda \right) -1}x_i=\sum_{j_0\in V_0}{\frac{1}{\sqrt{d_{n-1}\left( i \right) d_{n-1}\left( j_0 \right)}}x_{j_0}} $ can be obtained from Eq.(\ref{eq5}), and simplify it to obtain the following formula,
    \begin{equation}\label{eq6}
        \frac{-4\lambda ^3+12\lambda ^2-9\lambda +1}{2\lambda ^2-4\lambda +1}x_i=\sum_{j_0\in V_0}{\frac{1}{\sqrt{d_{n-1}\left( i \right) d_{n-1}\left( j_0 \right)}}x_{j_0}}.
    \end{equation}

    Thus, $ \frac{4\lambda ^3-10\lambda ^2+5\lambda }{2\lambda ^2-4\lambda +1}$ is an eigenvalue of $\mathcal{L}_{n-1}$  and $X_o=\left( x_i \right) _{i\in V_{old}}^{T}$ is the corresponding eigenvector. $X$ can be obtained from $X_o$ by Equations (\ref{pro1_leqs}). Then we have $m_{\mathcal{L} _{n-1}}\left(  \frac{4\lambda ^3-10\lambda ^2+5\lambda}{2\lambda ^2-4\lambda +1} \right)$ = $m_{\mathcal{L} _n}\left( \lambda \right)$, since there is a bijection between the eigenvectors of $\mathcal{L}_n$ and $\mathcal{L}_{n-1}$.

    %$X$ can be totally determined by $X_o$ using Equations (\ref{pro1_leqs}) and $m_{\mathcal{L} _{n-1}}\left(  \frac{4\lambda ^3-10\lambda ^2+5\lambda}{2\lambda ^2-4\lambda +1} \right) \geqslant m_{\mathcal{L} _n}\left( \lambda \right)$ .

    %In fact, $m_{\mathcal{L} _{n-1}}\left(  \frac{4\lambda ^3-10\lambda ^2+5\lambda}{2\lambda ^2-4\lambda +1} \right) = m_{\mathcal{L} _n}\left( \lambda \right)$. Otherwise there exists at least an extra eigenvector $X_{o}^{'}$ associated to $ \frac{4\lambda ^3-10\lambda ^2+5\lambda}{2\lambda ^2-4\lambda +1} $ without a corresponding eigenvector in $\mathcal{L}_{n}$. But Equations (\ref{pro1_leqs}) gives $X_{o}^{'}$ an associated eigenvector of $\mathcal{L}_{n}$, which is a contradiction with

    %$$m_{\mathcal{L} _{n-1}}\left(  \frac{4\lambda ^3-10\lambda ^2+5\lambda}{2\lambda ^2-4\lambda +1} \right) > m_{\mathcal{L} _n}\left( \lambda \right)$$.

    When $\lambda = \frac{2-\sqrt{2}}{2}$, $Det(\lambda) \ne 0$, Equations (\ref{pro1_leqs}) has a unique  solution, and we can get $x_{j_1}=-\frac{1}{\sqrt{d_{n-1}(j_0)}}x_{j_0}$. Combining it with Eq.(\ref{pro1_1}), we have $x_i = 0$ for any vertex $i\in V_{old}$.
    % $$
    %     \frac{\sqrt{2}}{2}x_i=-\sum_{j_0\in V_0}{\frac{1}{2\sqrt{d_{n-1}\left( i \right) d_{n-1}\left( j_0 \right)}}x_{j_0}}+\sum_{j_0\in V_0}{\frac{1}{2\sqrt{d_{n-1}\left( i \right) d_{n-1}\left( j_0 \right)}}x_{j_0}}=0.
    % $$
    Thus, Equations (\ref{pro1_leqs}) have only zero solution, that is, $X = 0$, and $\lambda = \frac{2-\sqrt{2}}{2}$ is not an eigenvalue of $\mathcal{L}_n$. The conclusions and proof  are the same when  $\lambda = \frac{2+\sqrt{2}}{2}$.

    Let $\sigma = \frac{4\lambda ^3-10\lambda ^2+5\lambda}{2\lambda ^2-4\lambda +1}$, when $\lambda =  \frac{5\pm \sqrt{5}}{4},\frac{3\pm \sqrt{5}}{4}$, we get $\sigma=0,2$. When $\lambda = \frac{1}{2}, \frac{3}{2}$, we get $\sigma=-1,3$. Obviously, $-1,3$ are not eigenvalues of $\mathcal{L}_{n-1}$.

    Therefore, if $\sigma $ is an eigenvalue of $ \mathcal{L}_{n-1} $ such that $\sigma  \ne 0,2$, then the roots of the equation $4\lambda ^3-\left( 10+2\sigma \right) \lambda ^2+\left( 5+4\sigma \right) \lambda -\sigma =0$ with $\lambda$ as variable are the eigenvalues of $\mathcal{L}_n$.

    %This completes the proof of   \ref{theo1_1}.

     \ref{theo1_2} It is obvious from  \cite{Spectralgraphtheory}.
    %It is obvious from Lemma\ref{lemma1}.

     \ref{theo1_3} When $\lambda= \frac{1}{2},\ Det(\lambda) \ne 0$, the Equations (\ref{pro1_leqs}) has a unique definite solution. Substituting $\lambda= \frac{1}{2}$ into Equations (\ref{pro1_leqs}), we get
    \begin{equation}\label{eq7}
        x_{j_1}=-\left( \frac{1}{\sqrt{d_{n-1}\left( j_0 \right)}}x_{j_0}-\frac{1}{\sqrt{d_{n-1}\left( i \right)}}x_i \right).
    \end{equation}
    We obtain by  substituting Eq.(\ref{eq7}) and $\lambda= \frac{1}{2}$ into Eq.(\ref{pro1_1}) that
    % $$
    %     \begin{aligned}
    %         \frac{1}{2}x_i=\sum_{j_0\in V_0}{\frac{1}{2d_{n-1}\left( i \right)}}x_i=\frac{1}{2}x_i.
    %     \end{aligned}      
    % $$
    %The above equality indicates
    Eq.(\ref{pro1_1}) is an identical equation. Therefore, the eigenvectors corresponding to $\lambda= \frac{1}{2}$ can be obtained by $x_i$ and $x_{j_0}$ according to Equations (\ref{pro1_leqs}). The conclusions and proof  are the same when  $\lambda = \frac{3}{2}$. Thus, we get $m_{\mathcal{L} _n}\left(  \frac{1}{2} \right)=m_{\mathcal{L} _n}\left(  \frac{3}{2} \right)=N_{n-1}$, where $n \geqslant 1$.

     \ref{theo1_4} Substituting $\lambda =  \frac{5 -\sqrt{5}}{4}$ into the Equations (\ref{pro1_leqs}) and simplify it as follows
    \begin{equation}\label{pro1_leqs2}
        \left( \begin{matrix}
                0  & 0  & 0  & 0                    \\
                -1 & 0  & 0  & -1                   \\
                0  & -1 & 0  & \frac{1-\sqrt{5}}{2} \\
                0  & 0  & -1 & \frac{\sqrt{5}-1}{2} \\
            \end{matrix} \right) \left( \begin{array}{c}
                x_{j_1} \\
                x_{j_2} \\
                x_{j_3} \\
                x_{j_4} \\
            \end{array} \right) =\left( \begin{array}{c}
                \frac{1}{\sqrt{d_{n-1}\left( i \right)}}x_i-\frac{1}{\sqrt{d_{n-1}\left( j_0 \right)}}x_{j_0} \\
                \frac{1-\sqrt{5}}{2}\frac{1}{\sqrt{d_{n-1}\left( j_0 \right)}}x_{j_0}                         \\
                \frac{\sqrt{5}-1}{2}\frac{1}{\sqrt{d_{n-1}\left( j_0 \right)}}x_{j_0}                         \\
                \frac{1}{\sqrt{d_{n-1}\left( j_0 \right)}}x_{j_0}                                             \\
            \end{array} \right) .
    \end{equation}
    Obviously, Equations (\ref{pro1_leqs2}) has solutions if and only if
    \begin{equation}\label{eq8}
        \frac{1}{\sqrt{d_{n-1}\left( i \right)}}x_i=\frac{1}{\sqrt{d_{n-1}\left( j_0 \right)}}x_{j_0}.
    \end{equation}
    Set $ \frac{1}{\sqrt{d_{n-1}\left( i \right)}}x_i=C$, $i \in V_{old}$ due to the connectivity of $H_{n-1}(G)$.

    Substituting Eq.(\ref{eq8}) into Eq.(\ref{pro1_1}), it follows
    % $$
    %     \begin{aligned}
    %         \frac{\sqrt{5}-1}{4}\sqrt{d_{n-1}(i)}C =\frac{1}{2\sqrt{d_{n-1}\left( i \right)}}\sum_{j_1\in V_1}{x_{j_1}}+\frac{1}{2}\sqrt{d_{n-1}(i)}C ,
    %     \end{aligned}
    % $$
    %that is
    \begin{equation}\label{eq12}
        \sum_{j_1\in V_1}{x_{j_1}}=\frac{\sqrt{5}-3}{2}d_{n-1}\left( i \right) C.
    \end{equation}

    From Equations (\ref{pro1_leqs2}), we get
    \begin{equation}\label{eq9}
        x_{j_1}=-x_{j_4}-\frac{1-\sqrt{5}}{2}C.
    \end{equation}
    Substituting Eqs.(\ref{eq8}) and (\ref{eq9}) into (\ref{pro1_1}), we have
    \begin{equation}\label{eq10}
        \sum_{j_4\in V_4}{x_{j_4}}=d_{n-1}\left( i \right) C.
    \end{equation}

    By Eq.(\ref{eq12}), the sum of elements of $X$ with respect to the newly added vertices of $H_n(G)$ adjacent to the vertices in $V_{old}$ is
    \begin{equation}
        \begin{aligned}\label{eq13}
            \sum_{i\in V_{old}}{\sum_{j\in V_{new},j\sim i}{x_j}}=\sum_{i\in V_{old}}{\frac{\sqrt{5}-3}{2}d_{n-1}\left( i \right) C}=\left( \sqrt{5}-3 \right) CE_{n-1}.
        \end{aligned}
    \end{equation}
    However, by Eq.(\ref{eq10}), we also have
    \begin{equation}
        \begin{aligned}\label{eq14}
            \sum_{i\in V_{old}}{\sum_{j\in V_{new},j\sim i}{x_j}}=\sum_{i\in V_{old}}{d_{n-1}\left( i \right) C}=2CE_{n-1} .
        \end{aligned}
    \end{equation}
    Thus, from Eqs.(\ref{eq13}) and (\ref{eq14}) we get $C = 0$, that is, for any $i \in V_{old}$ we have $x_i = 0$.

    Consequently, the eigenvector $X=(x_1,x_2,\cdots,x_{N_n})^T$  corresponding to $\lambda =  \frac{5 -\sqrt{5}}{4}$ can be completely obtained by equations as follows,
    \begin{equation}
        \begin{cases}
            \begin{matrix}
                x_i=0, & for\,\,all\,\,i\in V_{old}; \\
            \end{matrix}                                                                                                                                          \\
            \begin{matrix}
                \sum_{j_4\in V_4}{x_{j_4}}=0, & for\,\,all\,\,i\in V_{old}; \\
            \end{matrix}                                                                                            \\
            \begin{matrix}
                x_{j_1}=-x_{j_4},x_{j_2}=\frac{1-\sqrt{5}}{2}x_{j_4},x_{j_3}=\frac{\sqrt{5}-1}{2}x_{j_4}, & for\,\,all\,\,i,j_0\in V_{old}\,\,and\,\,i\sim j_0; \\
            \end{matrix}\,\, \\
        \end{cases}
    \end{equation}

    Suppose $ x_{j_1}=-x_{j_4}=y_j,\ j=1,2,\cdots,E_{n-1}$ for the newly added two vertices $j_1, j_4$ corresponding to an edge in $H_{n-1}(G)$, and let $Y=(y_j)^T$ which is a vector of dimension $E_{n-1}$. Note that $\sum_{j_4\in V_4}{x_{j_4}}=0$ for all $i \in V_{old}$ can be written as $\vec{B}\left( H_{n-1}\left( G \right) \right) Y=0$, where $\vec{B}\left( H_{n-1}\left( G \right) \right)$ is the directed
    incident matrix of $H_{n-1}\left( G \right)$.
    According to  \cite{directed},
    $
        rank\left( \vec{B}\left( H_{n-1}\left( G \right) \right) \right) =N_{n-1}-1
    $.
    Thus,  $\vec{B}\left( H_{n-1}\left( G \right) \right) Y=0$ has at most $E_{n-1}-N_{n-1}+1$ linearly independent solutions, i.e., $m_{\mathcal{L} _n}\left(  \frac{5-\sqrt{5}}{4} \right)=E_{n-1}-N_{n-1}+1$.

    When $\lambda= \frac{5+\sqrt{5}}{4}$, its conclusions and proofs are similar to that of $ \frac{5-\sqrt{5}}{4}$, thus they are omitted.

     \ref{theo1_5}-\ref{theo1_6}
    % Equations (\ref{pro1_leqs}) can be expressed as 
    % \begin{equation}\label{pro1_leqs3}
    %     \left( \begin{matrix}
    %         0&		0&		0&		0\\
    %         -1&		0&		0&		1\\
    %         0&		-1&		0&		\frac{\sqrt{5}+1}{2}\\
    %         0&		0&		-1&		\frac{\sqrt{5}+1}{2}\\
    %     \end{matrix} \right) \left( \begin{array}{c}
    %         x_{j_1}\\
    %         x_{j_2}\\
    %         x_{j_3}\\
    %         x_{j_4}\\
    %     \end{array} \right) =\left( \begin{array}{c}
    %         \frac{1}{\sqrt{d_{n-1}\left( i \right)}}x_i+\frac{1}{\sqrt{d_{n-1}\left( j_0 \right)}}x_{j_0}\\
    %         \frac{\sqrt{5}+1}{2}\frac{1}{\sqrt{d_{n-1}\left( j_0 \right)}}x_{j_0}\\
    %         \frac{\sqrt{5}+1}{2}\frac{1}{\sqrt{d_{n-1}\left( j_0 \right)}}x_{j_0}\\
    %         \frac{1}{\sqrt{d_{n-1}\left( j_0 \right)}}x_{j_0}\\
    %     \end{array} \right) .
    % \end{equation}
    % Similar to the proof of  \ref{theo1_4}, we get $x_i = 0$ for any $i \in V_{old}$.
    % Therefore, 
    %When $\lambda =  \frac{3 -\sqrt{5}}{4}$, 
    Similar to the proof of  \ref{theo1_4}, the eigenvector $X=(x_1,x_2,\cdots,x_{N_n})^T$  corresponding to $\lambda =  \frac{3 -\sqrt{5}}{4}$ can be determined by the following equations,
    \begin{equation}
        \begin{cases}
            \begin{matrix}
                x_i=0, & for\,\,all\,\,i\in V_{old}; \\
            \end{matrix}                                                                                              \\
            \begin{matrix}
                \sum_{j_4\in V_4}{x_{j_4}}=0, & for\,\,all\,\,i\in V_{old}; \\
            \end{matrix}                                                \\
            \begin{matrix}
                x_{j_1}=x_{j_4},x_{j_2}=x_{j_3}=\frac{\sqrt{5}+1}{2}x_{j_4}, & for\,\,all\,\,i,j_0\in V_{old}\,\,and\,\,i\sim j_0; \\
            \end{matrix}\,\, \\
        \end{cases}
    \end{equation}

    Suppose $ x_{j_1}=x_{j_4}=y_j,\ j=1,2,\cdots,E_{n-1}$ for the newly added two vertices $j_1, j_4$ corresponding to an edge in $H_{n-1}(G)$, and let $Y=(y_j)^T$. Hence,  $\sum_{j_4\in V_4}{x_{j_4}}=0$ for all $i \in V_{old}$ are equivalent to the equation system $B\left( H_{n-1}\left( G \right) \right) Y=0$, where $B\left( H_{n-1}\left( G \right) \right)$ is the incident matrix of $H_{n-1}\left( G \right)$.
    According to  \cite{incidence_matrix}, $rank\left( B\left( H_{n-1}\left( G \right) \right) \right)$ is equal to $N_{n-1}-1$ if $G$ is bipartite and $N_{n-1}$ otherwise. Thus, $B\left( H_{n-1}\left( G \right) \right) Y=0$ has at most $E_{n-1}-N_{n-1}+1$ linearly independent solutions if $G$ is bipartite and $E_{n-1}-N_{n-1}$ otherwise. Therefore, we have $m_{\mathcal{L} _n}\left(  \frac{3-\sqrt{5}}{4} \right)=E_{n-1}-N_{n-1}+1$ when $G$ is a bipartite graph and $m_{\mathcal{L} _n}\left(  \frac{3-\sqrt{5}}{4} \right)=E_{n-1}-N_{n-1}$ when $G$ is a non-bipartite graph.

    When $\lambda= \frac{3+\sqrt{5}}{4}$, its conclusions and proofs are similar to that of $ \frac{3-\sqrt{5}}{4}$, thus they are omitted.

    % \begin{remark}
    %     Since no new odd cycles are generated in the construction of hexagonal graphs, if $G$ is bipartite, then $H_n(G)$ is bipartite, otherwise $H_n(G)$ is non-bipartite, where $n \geqslant 1$.
    % \end{remark}

    This complete the proof of Theorem \ref{theo1}. 
\end{proof}

\subsection{The normalized Laplacian spectrum of  \texorpdfstring{$H^k_{n}(G)$}{}}

\begin{theorem} \label{theo2}

    Denote by $\mathcal{L}^k_n $  the normalized Laplacian matrix of $H^k_{n}(G)$. The normalized Laplacian spectrum of $H^k_n(G)$, $n \geqslant 1$ is given below.

    % Let $G$ be a simple connected graph with $N^k_0$ vertices and $E^k_0$ edges, and $H^k_{n}(G)$ be the $n$-th $k$-hexagonal graph of $G$, where $k \geqslant 2$. The normalized Laplacian matrix of $H^k_{n}(G)$  is denoted by $\mathcal{L}^k_n $.
    % The normalized Laplacian spectrum of $H^k_n(G)$ can be obtained as follows.

    \begin{enumerate}[label=(\roman*)]
        \item \label{theo2_1}If $\sigma $ is an eigenvalue of $ \mathcal{L}^k_{n-1} $ satisfying $\sigma  \ne 0,2$, then $\lambda_i,\ i=1,2,3,4,5$ are the eigenvalues of $ \mathcal{L}^k_{n} $ with $m_{\mathcal{L}^k_{n}}(\lambda_i)=m_{\mathcal{L}^k_{n-1}}(\sigma)$, where $\lambda_i,\ i=1,2,3,4,5$ are the roots of
              \begin{equation}\label{eq19}
                  \begin{aligned}
                      \left( 16k+16 \right) \lambda ^5-\left( 80k+64+16\sigma \right) \lambda ^4+\left( 140k+84+64\sigma \right) \lambda ^3 \\
                      -\left( 100k+40+84\sigma \right) \lambda ^2
                      +\left( 25k+5+40\sigma \right) \lambda -\sigma \left( k+5 \right) =0.
                  \end{aligned}
              \end{equation}

        \item \label{theo2_2} 0, $ \frac{5k+3-\sqrt{5k^2+6k+5}}{4\left( k+1 \right)}$ and $ \frac{5k+3+\sqrt{5k^2+6k+5}}{4\left( k+1 \right)}$ are the eigenvalues of $\mathcal{L}^k_{n}$ with the multiplicity 1.

        \item \label{theo2_3} When $G$ is bipartite, 2, $ \frac{3k+5-\sqrt{5k^2+6k+5}}{4\left( k+1 \right)}$ and $ \frac{3k+5+\sqrt{5k^2+6k+5}}{4\left( k+1 \right)}$ are the eigenvalues of $\mathcal{L}^k_{n}$ with the multiplicity 1.

        \item \label{theo2_4} $ \frac{5-\sqrt{5}}{4}$ and $ \frac{5+\sqrt{5}}{4}$  are the eigenvalues of $\mathcal{L}^k_n$ with multiplicity $kE^k_{n-1}-N^k_{n-1}+1$.

        \item \label{theo2_5} When $G$ is bipartite, $ \frac{3-\sqrt{5}}{4}$ and $ \frac{3+\sqrt{5}}{4}$  are the eigenvalues of $\mathcal{L}^k_n$ with multiplicity $kE^k_{n-1}-N^k_{n-1}+1$.

        \item \label{theo2_6} When $G$ is non-bipartite, $ \frac{3-\sqrt{5}}{4}$ and $ \frac{3+\sqrt{5}}{4}$  are the eigenvalues of $\mathcal{L}^k_n$ with multiplicity $kE^k_{n-1}-N^k_{n-1}$.
    \end{enumerate}
\end{theorem} 

\begin{proof}\rm
    Partition the set of vertices in $H_n(G)$ as $V_{old} \cup V_{new}$, in which $V_{old}$ is the set of all the vertices inherited from $H_{n-1}(G)$, and $V_{new}$ is its complement. For any vertex $i\in V_{old}$, let $ V_0(i)\subset V_{old}$ be the set of old vertices in $H_n(G)$ adjacent to $i$,  and for each vertex,
    $$
        ij_{1}^{1}j_{2}^{1}j_{3}^{1}j_{4}^{1}j_0,\ ij_{1}^{2}j_{2}^{2}j_{3}^{2}j_{4}^{2}j_0,\ \cdots \,\,,\ ij_{1}^{k}j_{2}^{k}j_{3}^{k}j_{4}^{k}j_0,
    $$
    are the $k$ newly added parallel paths of length 5 in $H^k_n(G)$ with respect to edge $ij_0$.

    Then we let $V_{1}^{l}\left( i \right) =\left\{ j_{1}^{l}:\forall j_0\in V_0(i) \right\}$ denote the set of second vertices $j_1^l$ in the $l-th$ path $ij_{1}^{l}j_{2}^{l}j_{3}^{l}j_{4}^{l}j_0$ corresponding to edge $ij_0 $ for each vertex $j_{0} \in V_{0}(i)$, where $l=1,2,\cdots,k$. For convenience, we use $V_0$ to denote $V_0(i)$ and $V_1$ to denote $V_1(i)$. Similarly, let $V_{2}^{l}=\left\{ j_{2}^{l}:\forall j_0\in V_0 \right\},\ V_{3}^{l}=\left\{ j_{3}^{l}:\forall j_0\in V_0 \right\}$ and $V_{4}^{l}=\left\{ j_{4}^{l}:\forall j_0\in V_0 \right\}$.

    % Clearly, we have
    % $$
    % V_{new}=\bigcup_{i\in V_{old}}{\bigcup_{l=1}^k{\bigcup_{t=1}^4{V_{t}^{l}}}}
    % $$

    Let $X=(x_1,x_2,\cdots,x_{N^k_n})^T$ be an eigenvector with respect to the eigenvalue $\lambda$ of $\mathcal{L}^k_n$ and $d_n(v)$ be the degree of vertex $v$ of $H^k_n(G)$, then we have

    $$
        d_n\left( v \right) \begin{cases}
            (k+1)d_{n-1}\left( v \right), \begin{matrix}
                                              & v\in V_{old}; \\
                                         \end{matrix}             \\
            2, \begin{matrix}
                  &  & \begin{matrix}
                           &  & \\
                      \end{matrix} & \qquad\ \		v\in V_{new}. \\
             \end{matrix} \\
        \end{cases}
    $$

    For any vertex $i\in V_{old}$ and $l\in \left\{ 1,2,\cdots ,k \right\}$, Eq.(\ref{eq3}) leads to
    \begin{equation}\label{pro2_1}
        \begin{aligned}
            \left( 1-\lambda \right) x_i=\sum_{l=1}^k{\sum_{j_{1}^{l}\in V_{1}^{l}}{\frac{1}{\sqrt{2\left( k+1 \right) d_{n-1}\left( i \right)}}x_{j_{1}^{l}}}}+\sum_{j_0\in V_0}{\frac{1}{(k+1)\sqrt{d_{n-1}\left( i \right) d_{n-1}\left( j_0 \right)}}x_{j_0}}.
        \end{aligned}
    \end{equation}

    Given vertex $j_0 \in V_{0}$ and $l\in \left\{ 1,2,\cdots ,k \right\}$, we consider the vertices of the $l-th$ path on edge $ij_0$, indicated by $ij_{1}^{l}j_{2}^{l}j_{3}^{l}j_{4}^{l}j_0$.

    For the vertex $j_1^l \in V_{1}^l$, we have
    \begin{equation}\label{pro2_2}
        \left( 1-\lambda \right) x_{j_{1}^{l}}=\frac{1}{\sqrt{2\left( k+1 \right) d_{n-1}\left( i \right)}}x_i+\frac{1}{2}x_{j_{2}^{l}}.
    \end{equation}

    Analogously, for the vertex $j_2^l \in V_{2}^l$, we obtain
    \begin{equation}\label{pro2_3}
        \left( 1-\lambda \right) x_{j_{2}^{l}}=\frac{1}{2}x_{j_{1}^{l}}+\frac{1}{2}x_{j_{3}^{l}}.
    \end{equation}

    For the vertex $j_3^l \in V_{3}^l$, we have
    \begin{equation}\label{pro2_4}
        \left( 1-\lambda \right) x_{j_{3}^{l}}=\frac{1}{2}x_{j_{2}^{l}}+\frac{1}{2}x_{j_{4}^{l}}.
    \end{equation}

    For the vertex $j_4^l \in V_{4}^l$, we get
    \begin{equation}\label{pro2_5}
        \left( 1-\lambda \right) x_{j_{4}^{l}}=\frac{1}{2}x_{j_{3}^{l}}+\frac{1}{\sqrt{2\left( k+1 \right) d_{n-1}\left( j_0 \right)}}x_{j_0}.
    \end{equation}

    Combining Eqs.(\ref{pro2_2}), (\ref{pro2_3}), (\ref{pro2_4}) and (\ref{pro2_5}), we get linear equations with $x_{j_1},x_{j_2},x_{j_3}, x_{j_4}$ as variables in matrix form, as follows
    \begin{equation}\label{pro2_leqs}
        \left( \begin{matrix}
                2\left( 1-\lambda \right) & -1                        & 0                         & 0                         \\
                -1                        & 2\left( 1-\lambda \right) & -1                        & 0                         \\
                0                         & -1                        & 2\left( 1-\lambda \right) & -1                        \\
                0                         & 0                         & -1                        & 2\left( 1-\lambda \right) \\
            \end{matrix} \right) \left( \begin{array}{c}
                x_{j_{1}^{l}} \\
                x_{j_{2}^{l}} \\
                x_{j_{3}^{l}} \\
                x_{j_{4}^{l}} \\
            \end{array} \right) =\left( \begin{array}{c}
                \frac{1}{\sqrt{\frac{k+1}{2}d_{n-1}\left( i \right)}}x_i       \\
                0                                                              \\
                0                                                              \\
                \frac{1}{\sqrt{\frac{k+1}{2}d_{n-1}\left( j_0 \right)}}x_{j_0} \\
            \end{array} \right) .
    \end{equation}

    The determinant of coefficient matrix of Equations (\ref{pro2_leqs}) is denoted by $Det(\lambda)$, that is
    $$
        Det\left( \lambda \right) =16\lambda ^4-64\lambda ^3+84\lambda ^2-40\lambda +5.
    $$

     \ref{theo2_1} When $\lambda \ne  \frac{5\pm \sqrt{5}}{4},\frac{3\pm \sqrt{5}}{4}$, we have $Det(\lambda) \ne 0$. Equations (\ref{pro2_leqs}) has a unique  solution  and  $x_{j_{t}^{1}}=x_{j_{t}^{2}}=\cdots =x_{j_{t}^{k}}$, where $t=1,2,3,4$. $x_{j^l_1}$ can be determined by
    \begin{equation}\label{eq15}
        x_{j_{1}^{l}}=\frac{1}{Det\left( \lambda \right)}\left( \frac{1}{\sqrt{\frac{k+1}{2}d_{n-1}\left( j_0 \right)}}x_{j_0}-\frac{8\lambda ^3-24\lambda ^2+20\lambda -4}{\sqrt{\frac{k+1}{2}d_{n-1}\left( i \right)}}x_i \right) .
    \end{equation}

    %In view of Equations (\ref{pro2_leqs}), we have $x_{j_{t}^{1}}=x_{j_{t}^{2}}=\cdots =x_{j_{t}^{k}}$, where $t=1,2,3,4$.
    Substituting Eq.(\ref{eq15}) into Eq.(\ref{pro2_1}) yields
    \begin{equation}\label{eq16}
        \begin{aligned}
             & \left( 1-\lambda \right) x_i                                                                                                                                                                                                                                                              \\
             & \quad=\left( \frac{k}{Det\left( \lambda \right)}+1 \right) \sum_{j_0\in V_0}{\frac{1}{\left( k+1 \right) \sqrt{d_{n-1}\left( i \right) d_{n-1}\left( j_0 \right)}}x_{j_0}}-\frac{k\left( 8\lambda ^3-24\lambda ^2+20\lambda -4 \right)}{\left( k+1 \right) Det\left( \lambda \right)}x_i.
        \end{aligned}
    \end{equation}

    Note that $Det(\lambda)+k>0$ for any $k \geqslant 2$. Then by Eq.(\ref{eq16}) we have
    \begin{equation}\label{eq17}
        \left( 1-\frac{\varphi \left( \lambda \right)}{Det\left( \lambda \right) +k} \right) x_i=\sum_{j_0\in V_0}{\frac{1}{ \sqrt{d_{n-1}\left( i \right) d_{n-1}\left( j_0 \right)}}x_{j_0}},
    \end{equation}
    where $\varphi \left( \lambda \right) =\left( 16k+16 \right) \lambda ^5-\left( 80k+64 \right) \lambda ^4+\left( 140k+84 \right) \lambda ^3-\left( 100k+40 \right) \lambda ^2+\left( 25k+5 \right)
    $.

    %$\left( 1-\frac{\varphi \left( \lambda \right)}{Det\left( \lambda \right) +k} \right)$ is an eigenvalue of $P^k_n$. Thus

    Eq.(\ref{eq17}) shows that $ \frac{\varphi \left( \lambda \right)}{Det\left( \lambda \right) +k}$ is an eigenvalue of $\mathcal{L}_{n-1}^k$  and $X_o=\left( x_i \right) _{i\in V_{old}}^{T}$ is one associated eigenvector.
    $X$ can be totally determined by $X_o$ using Equations (\ref{pro2_leqs}). Then we have $m_{\mathcal{L}^k _{n-1}}\left(  \frac{\varphi \left( \lambda \right)}{Det\left( \lambda \right) +k} \right) = m_{\mathcal{L}^k _n}\left( \lambda \right)$, since there is a bijection  between the eigenvectors of $\mathcal{L}_n$ and $\mathcal{L}_{n-1}$.

    %In fact, $m_{\mathcal{L}^k _{n-1}}\left(  \frac{\varphi \left( \lambda \right)}{Det\left( \lambda \right) +k} \right) = m_{\mathcal{L} _n^k}\left( \lambda \right)$. Otherwise there exists at least an extra eigenvector $X_{o}^{'}$ associated to $ \frac{\varphi \left( \lambda \right)}{Det\left( \lambda \right) +k} $ without a corresponding eigenvector in $\mathcal{L}^k_{n}$. But Equations (\ref{pro2_leqs}) gives $X_{o}^{'}$ an associated eigenvector of $\mathcal{L}^k_{n}$, which is a contradiction with $m_{\mathcal{L}^k _{n-1}}\left(  \frac{\varphi \left( \lambda \right)}{Det\left( \lambda \right) +k} \right) > m_{\mathcal{L}^k _n}\left( \lambda \right)$.

    Let $\sigma = \frac{\varphi \left( \lambda \right)}{Det\left( \lambda \right) +k}$, when $\lambda =  \frac{5\pm \sqrt{5}}{4},\frac{3\pm \sqrt{5}}{4}$, we get $\sigma=0,2$. Therefore, if $\sigma $ is an eigenvalue of $ \mathcal{L}^k_{n-1} $ such that $\sigma  \ne 0,2$, then the roots of the Eq.(\ref{eq19}) with $\lambda$ as variable are the eigenvalues of $\mathcal{L}^k_n$.

     \ref{theo2_2}-\ref{theo2_3}
    Suppose $\lambda$ is an eigenvalue of $\mathcal{L}^k_n$ satisfying $\lambda \ne  \frac{5\pm \sqrt{5}}{4},\frac{3\pm \sqrt{5}}{4}$. By the proof of  \ref{theo2_1}, we have that
    $$
        \sigma =\frac{\varphi \left( \lambda \right)}{Det\left( \lambda \right) +k}
    $$
    is an eigenvalue $\mathcal{L}^k_{n-1}$ that has the same multiplicity as $\lambda$. Hence, if $\lambda$  is a root of Eq.(\ref{eq19}) such that $\lambda \ne \frac{5\pm \sqrt{5}}{4},\frac{3\pm \sqrt{5}}{4}$, then $\lambda$ is an eigenvalue of $\mathcal{L}^k_n$.
    % \begin{equation}\label{eq20}
    %     \begin{aligned}
    %     \left( 16k+16 \right) \lambda ^5-\left( 80k+64+16\sigma \right) \lambda ^4+\left( 140k+80+64\sigma \right) \lambda ^3\\
    %     -\left( 100k+40+84\sigma \right) \lambda ^2
    %     +\left( 25k+5+40\sigma \right) \lambda -\sigma \left( k+5 \right) =0.    
    %     \end{aligned}  
    % \end{equation}

    Substituting $\sigma=0$ into Eq.(\ref{eq19}) yields $\lambda=0, \frac{5k+3\pm \sqrt{5k^2+6k+5}}{4\left( k+1 \right)}$ and substituting $\sigma=2$ into Eq.(\ref{eq19}) yields $\lambda=2, \frac{3k+5\pm \sqrt{5k^2+6k+5}}{4\left( k+1 \right)}$.

     \ref{theo2_4} Similar to the proof of Theorem \ref{theo1}\ref{theo1_4}, the eigenvector $X=(x_1,x_2,\cdots,x_{N^k_n})^T$  corresponding to $\lambda =  \frac{5 -\sqrt{5}}{4}$ can be completely obtained by equations as follows,
    $$
        \begin{cases}
            \begin{matrix}
                x_i=0, & for\,\,all\,\,i\in V_{old}; \\
            \end{matrix}                                                                                                                                                                                               \\
            \begin{matrix}
                \sum_{l=1}^k{\sum_{j_{4}^{l}\in V_{4}^{l}}{x_{j_{4}^{l}}}}=0, & for\,\,all\,\,i\in V_{old}; \\
            \end{matrix}                                                                             \\
            \begin{matrix}
                x_{j_{1}^{l}}=-x_{j_{4}^{l}},x_{j_{2}^{l}}=\frac{1-\sqrt{5}}{2}x_{j_{4}^{l}},x_{j_{3}^{l}}=\frac{\sqrt{5}-1}{2}x_{j_{4}^{l}}, & for\,\,all\,\,i,j_0\in V_{old}\,\,and\,\,i\sim j_0, l=1,2,\cdots ,k; \\
            \end{matrix}\,\, \\
        \end{cases}
    $$

    Suppose
    $
        % \begin{aligned}
        %     x_{j_{1}^{1}}=-x_{j_{4}^{1}}=y_{j}^{1}&,\  j=1,2,\cdots E^k_{n-1};\\
        %     x_{j_{1}^{2}}=-x_{j_{4}^{2}}=y_{j}^{2}&,\  j=1,2,\cdots E^k_{n-1};\\
        %     &\vdots \\
        %     x_{j_{1}^{k}}=-x_{j_{4}^{k}}=y_{j}^{k}&,\  j=1,2,\cdots E^k_{n-1}.
        % \end{aligned}
        x_{j_{1}^{l}}=-x_{j_{4}^{l}}=y_{j}^{l}, j=1,2,\cdots E_{n-1}^{k},l=1,2,\cdots ,k,
    $ for the newly added two vertices $j^l_1, j^l_4$ corresponding to an edge in $H^k_{n-1}(G)$.

    Let $Y=\left( y_{1}^{1},\cdots ,y_{E^k_{n-1}}^{1},y_{1}^{2},\cdots ,y_{E^k_{n-1}}^{2},\cdots ,y_{1}^{k},\cdots ,y_{E^k_{n-1}}^{k} \right) ^T$ which is a vector of dimension $kE^k_{n-1}$. Note that $\sum_{l=1}^k{\sum_{j_{4}^{l}\in V_{4}^{l}}{x_{j_{4}^{l}}}}=0$ for all $i \in V_{old}$ can be written as $\vec{W}Y=0$, where $\vec{W}=\left( \vec{B}\left( H_{n-1}^{k}\left( G \right) \right) ,\vec{B}\left( H_{n-1}^{k}\left( G \right) \right) ,\cdots ,\vec{B}\left( H_{n-1}^{k}\left( G \right) \right) \right)$ is a matrix with dimension $N_{n-1}^{k}\times kE_{n-1}^{k}$ and $\vec{B}\left( H^k_{n-1}\left( G \right) \right)$ is the directed incident matrix of $H^k_{n-1}\left( G \right)$. We know \newline
    $rank(\vec{W})=rank(\vec{B}\left( H^k_{n-1}\left( G \right) \right)) =N^k_{n-1}-1$, thus $\vec{W}Y=0$ has at most $kE^k_{n-1}-N^k_{n-1}+1$ linearly independent solutions, then we have $m_{\mathcal{L}^k _n}\left(  \frac{5-\sqrt{5}}{4} \right)=kE^k_{n-1}-N^k_{n-1}+1$.

    % when $\lambda= \frac{5+\sqrt{5}}{4}$, its conclusions and proofs are similar to that of $ \frac{5-\sqrt{5}}{4}$, thus they are omitted.
    %  When $\lambda =  \frac{3 -\sqrt{5}}{4}$, Equations (\ref{pro2_leqs}) can be expressed as 
    % \begin{equation}\label{pro2_leqs3}
    %     \left( \begin{matrix}
    %         0&		0&		0&		0\\
    %         -1&		0&		0&		1\\
    %         0&		-1&		0&		\frac{\sqrt{5}+1}{2}\\
    %         0&		0&		-1&		\frac{\sqrt{5}+1}{2}\\
    %     \end{matrix} \right) \left( \begin{array}{c}
    %         x_{j_{1}^{l}}\\
    %         x_{j_{2}^{l}}\\
    %         x_{j_{3}^{l}}\\
    %         x_{j_{4}^{l}}\\
    %     \end{array} \right) =\left( \begin{array}{c}
    %         \frac{1}{\sqrt{\frac{k+1}{2}d_{n-1}\left( i \right)}}x_i+\frac{1}{\sqrt{\frac{k+1}{2}d_{n-1}\left( j_0 \right)}}x_{j_0}\\
    %         \frac{\sqrt{5}+1}{2}\frac{1}{\sqrt{\frac{k+1}{2}d_{n-1}\left( j_0 \right)}}x_{j_0}\\
    %         \frac{\sqrt{5}+1}{2}\frac{1}{\sqrt{\frac{k+1}{2}d_{n-1}\left( j_0 \right)}}x_{j_0}\\
    %         \frac{1}{\sqrt{\frac{k+1}{2}d_{n-1}\left( j_0 \right)}}x_{j_0}\\
    %     \end{array} \right) .
    % \end{equation}
    % Therefore,  we get $x_i = 0$ for any $i \in V_{old}$.

     \ref{theo2_5}-\ref{theo2_6} Similar to the proof of Theorem \ref{theo1}\ref{theo1_4}, the eigenvector $X=(x_1,x_2,\cdots,x_{N^k_n})^T$  corresponding to $\lambda =  \frac{3 -\sqrt{5}}{4}$ can be completely obtained by equations as follows,
    $$
        \begin{cases}
            \begin{matrix}
                x_i=0, & for\,\,all\,\,i\in V_{old}; \\
            \end{matrix}                                                                                                                                            \\
            \begin{matrix}
                \sum_{l=1}^k{\sum_{j_{4}^{l}\in V_{4}^{l}}{x_{j_{4}^{l}}}}=0, & for\,\,all\,\,i\in V_{old}; \\
            \end{matrix}                       \\
            \begin{matrix}
                x_{j_{1}^{l}}=x_{j_{4}^{l}},x_{j_{2}^{l}}=x_{j_{3}^{l}}=\frac{\sqrt{5}+1}{2}x_{j_{4}^{l}}, & for\,\,all\,\,i,j_0\in V_{old}\,\,and\,\,i\sim j_0,l=1,2,\cdots ,k; \\
            \end{matrix}\,\, \\
        \end{cases}
    $$

    Suppose $x_{j_{1}^{l}}=x_{j_{4}^{l}}=y_{j}^{l}, j=1,2,\cdots E_{n-1}^{k},l=1,2,\cdots ,k,$ for the newly added two vertices $j^l_1, j^l_4$ corresponding to an edge in $H^k_{n-1}(G)$.
    % $$
    % \begin{aligned}
    %     % x_{j_{1}^{1}}=x_{j_{4}^{1}}=y_{j}^{1}&,\  j=1,2,\cdots E^k_{n-1};\\
    %     % x_{j_{1}^{2}}=x_{j_{4}^{2}}=y_{j}^{2}&,\  j=1,2,\cdots E^k_{n-1};\\
    %     % &\vdots \\
    %     % x_{j_{1}^{k}}=x_{j_{4}^{k}}=y_{j}^{k}&,\  j=1,2,\cdots E^k_{n-1}.
    % \end{aligned}
    % $$

    Let $Y=\left( y_{1}^{1},\cdots ,y_{E^k_{n-1}}^{1},y_{1}^{2},\cdots ,y_{E^k_{n-1}}^{2},\cdots ,y_{1}^{k},\cdots ,y_{E^k_{n-1}}^{k} \right) ^T$ which is a $kE^k_{n-1}$ dimensional vector. Note that $\sum_{l=1}^k{\sum_{j_{4}^{l}\in V_{4}^{l}}{x_{j_{4}^{l}}}}=0$ for all $i \in V_{old}$ can be written as $WY=0$, where $W=\left( B\left( H_{n-1}^{k}\left( G \right) \right) ,B\left( H_{n-1}^{k}\left( G \right) \right) ,\cdots ,B\left( H_{n-1}^{k}\left( G \right) \right) \right)$ is a matrix with dimension $N_{n-1}^{k}\times kE_{n-1}^{k}$
    and $B\left( H^k_{n-1}\left( G \right) \right)$ is the  incident matrix of $H^k_{n-1}\left( G \right)$. We know $rank(W)=rank(B\left( H^k_{n-1}\left( G \right) \right))=N^k_{n-1}-1$ if $G$ is bipartite and $N^k_{n-1}$ otherwise. Thus, $WY=0$ has at most $kE^k_{n-1}-N^k_{n-1}+1$ linearly independent solutions if $G$ is a bipartite graph, otherwise $kE^k_{n-1}-N^k_{n-1}$.  Therefore, we have $m_{\mathcal{L}^k _n}\left(  \frac{3-\sqrt{5}}{4} \right)=kE^k_{n-1}-N^k_{n-1}+1$ when $G$ is a bipartite graph and $m_{\mathcal{L}^k _n}\left(  \frac{3-\sqrt{5}}{4} \right)=kE^k_{n-1}-N^k_{n-1}$ when $G$ is a non-bipartite graph.

    When $\lambda= \frac{3+\sqrt{5}}{4}$, its conclusions and proofs are similar to that of $ \frac{3-\sqrt{5}}{4}$, thus they are omitted.

    This complete the proof of Theorem \ref{theo2}. 
\end{proof}

\section{Applications}
%From the normalized Laplacian spectrum of the $n$-th  hexagonal graph $H_{n}(G)$ (resp. $n$-th $k$-hexagonal graph $H^k_{n}(G)$), we compute certain structure-related invariants

In this section, We will provide complete formulas to calculate the Kemeny's constant,  multiplicative degree-Kirchhoff index, and the number of spanning trees of $H_{n}(G)$  (resp. $H^k_{n}(G)$).

\subsection{The Kemeny's constant, multiplicative degree-Kirchhoff index  and spanning trees of \texorpdfstring{$H_{n}(G)$}{}}
Let $f_1(\sigma), f_2(\sigma)$ and $f_3(\sigma)$ be the three roots of Eq.(\ref{eq18}). By Vieta's Theorem, we obtain
\begin{equation}\label{eq27}
    \frac{1}{f_1\left( \sigma \right)}+\frac{1}{f_2\left( \sigma \right)}+\frac{1}{f_2\left( \sigma \right)}=\frac{5+4\sigma}{\sigma},\ f_1\left( \sigma \right) f_2\left( \sigma \right) f_3\left( \sigma \right) =\frac{\sigma}{4}.
\end{equation}

Following are three multisets defined by a multiset $M$,
$$
    f_1\left( M \right) =\bigcup_{\chi \in M}{\left\{ f_1\left( \chi \right) \right\}},\ f_2\left( M \right) =\bigcup_{\chi \in M}{\left\{ f_2\left( \chi \right) \right\}},\ f_3\left( M \right) =\bigcup_{\chi \in M}{\left\{ f_3\left( \chi \right) \right\}}.
$$

\begin{theorem} \label{theo43}
    The Kemeny's constant $K(H_{n-1}(G))$ of $H_n(G)$ can be obtained by iteration of the following formula,
    $$
        \begin{aligned}
            K\left( H_n\left( G \right) \right) =5K\left( H_{n-1}\left( G \right) \right) -\frac{4}{3}N_0+\left( \frac{104}{15}6^{n-1}+\frac{16}{15} \right) E_0-2.
        \end{aligned}
    $$
    In general,
    $$
        \begin{aligned}
            K\left( H_n\left( G \right) \right) & =5^nK\left( G \right) -\frac{1}{3}\left( 5^n-1 \right) N_0+\frac{104}{3}\left[ \left( \frac{6}{5} \right) ^n-1 \right] \cdot 5^{n-1}E_0 \\
                                                & \quad+\frac{4}{15}\left( 5^n-1 \right) E_0
            -\frac{1}{2}\left( 5^n-1 \right) .
        \end{aligned}
    $$
\end{theorem} 
\begin{proof}\rm
    Denote the normalized Laplacian spectrum of $H_{n}(G)$ by $0 = \lambda^{(n)}_1 < \lambda^{(n)}_2 \leqslant \cdots \leqslant \lambda^{(n)}_{N_{n}}\leqslant 2$. According to  \cite{Algebraic_aspects_NL}, $
        K\left( H_n(G) \right) =\sum_{i=2}^{N_n}{\frac{1}{\lambda^{(n)}_i}}
    $ and note that $f_1\left( 2 \right) =\frac{3+\sqrt{5}}{2},f_2\left( 2 \right) =\frac{3-\sqrt{5}}{2},f_3\left( 2 \right) =2$. Whether $G$ is bipartite, we have  Eq.(\ref{eq27}) and Theorem \ref{theo1},
    $$
        \begin{aligned}
            K\left( H_n\left( G \right) \right) & =\sum_{i=2}^{N_{n-1}}{\left( \frac{5+4\lambda^{(n-1)}_i}{\lambda^{(n-1)}_i} \right)}+8E_{n-1}-\frac{4}{3}N_{n-1}-2  \\
                                                & =5K\left( H_{n-1}\left( G \right) \right) -\frac{4}{3}N_0+\left( \frac{104}{15}6^{n-1}+\frac{16}{15} \right) E_0-2.
        \end{aligned}
    $$
    From the above equation we obtain
    $$
        \begin{aligned}
            K\left( H_n\left( G \right) \right) & =5^nK\left( G \right) +\sum_{i=0}^{n-1}{5^{n-1-i}\left[ -\frac{4}{3}N_0+\left( \frac{104}{15}6^i+\frac{16}{15} \right) E_0-2 \right]}   \\
                                                & =5^nK\left( G \right) -\frac{1}{3}\left( 5^n-1 \right) N_0+\frac{104}{3}\left[ \left( \frac{6}{5} \right) ^n-1 \right] \cdot 5^{n-1}E_0
            \\&\quad
            +\frac{4}{15}\left( 5^n-1 \right) E_0-\frac{1}{2}\left( 5^n-1 \right) .\qquad\qquad \qquad \qquad \qquad \qquad 
        \end{aligned}
    $$
    %This result is a direct consequence of Eq.(\ref{eq28}) and Theorem \ref{theo42}. 
\end{proof}

\begin{theorem} \label{theo42}
    The relationship between the multiplicative degree-Kirchhoff indices of $H_n(G)$ and $H_{n-1}(G)$ is
    $$
        \begin{aligned}
            Kf^{'}\left( H_n\left( G \right) \right) & =30Kf^{'}\left( H_{n-1}\left( G \right) \right) -16\cdot 6^{n-1}E_0N_0+\left( \frac{416}{5}\cdot 6^{2n-2}+\frac{64}{5}\cdot 6^{n-1} \right) E_{0}^{2} \\
                                                     & \quad-24\cdot 6^{n-1}E_0.
        \end{aligned}
    $$
    Furthermore,
    \begin{equation}\label{eq29}
        \begin{aligned}
            Kf^{'}\left( H_n\left( G \right) \right) & =30^nKf^{'}\left( G \right) -4\cdot 6^{n-1}\left( 5^n-1 \right) E_0N_0+\frac{16}{5}\cdot 6^{n-1}\left( 5^n-1 \right) E_{0}^{2} \\
                                                     & \quad+416\cdot 30^{n-1}\left[ \left( \frac{6}{5} \right) ^n-1 \right] E_{0}^{2}-6\cdot 6^{n-1}\left( 5^n-1 \right) E_0.
        \end{aligned}
    \end{equation}
\end{theorem} 
\begin{proof}\rm
    % $$
    % \begin{aligned}
    %     Kf^{'}\left( H_n\left( G \right) \right) 
    %     &=12E_{n-1}\left[ \sum_{i=2}^{N_{n-1}}{\left( \frac{5+4\sigma _i}{\sigma _i} \right)}+8E_{n-1}-\frac{4}{3}N_{n-1}-2 \right] \\
    %     &=30Kf^{'}\left( H_{n-1}\left( G \right) \right) -16\cdot 6^{n-1}E_0N_0+\left( \frac{416}{5}\cdot 6^{2n-2}+\frac{64}{5}\cdot 6^{n-1} \right) E_{0}^{2}\\
    %     &\quad-24\cdot 6^{n-1}E_0.
    % \end{aligned}
    % $$
    % Therefore, from the above equation we obtain
    % $$
    % \begin{aligned}
    %     Kf^{'}\left( H_n\left( G \right) \right) &=30^nKf^{'}\left( G \right) \\
    %     &\quad+\sum_{i=0}^{n-1}{30^{n-1-i}\left[ -16\cdot 6^iE_0N_0+\left( \frac{416}{5}\cdot 6^{2i}+\frac{64}{5}\cdot 6^i \right) E_{0}^{2}-24\cdot 6^iE_0 \right]}\\
    %     &=30^nKf^{'}\left( G \right) -4\cdot 6^{n-1}\left( 5^n-1 \right) E_0N_0+\frac{16}{5}\cdot 6^{n-1}\left( 5^n-1 \right) E_{0}^{2}\\
    %     &\quad+416\cdot 30^{n-1}\left[ \left( \frac{6}{5} \right) ^n-1 \right] E_{0}^{2}-6\cdot 6^{n-1}\left( 5^n-1 \right) E_0.
    % \end{aligned}
    % $$
    % The proof is completed.
    According to  \cite{Algebraic_aspects_NL} and  \cite{1}, we have $Kf^{'}\left( H_n(G) \right) =2E_n\cdot K\left( H_n(G) \right)$. Therefore, this result can be obtained from Theorem \ref{theo43}. 
\end{proof}

\begin{theorem} \label{404}
    The number of spanning trees of $H_n\left( G \right)$ is
    \begin{equation}\label{eq30}
        \begin{aligned}
            \tau \left( H_n\left( G \right) \right) =5^{\left( \frac{1}{5}\cdot 6^{n-1}+\frac{4}{5} \right) E_0-N_0+1}\cdot 6^{N_0+\frac{4}{5}\left( 6^{n-1}-1 \right) E_0-1}\cdot \tau \left( H_{n-1}\left( G \right) \right) .
        \end{aligned}
    \end{equation}
    The general expression is
    $$
        \begin{aligned}
            \tau \left( H_n\left( G \right) \right) =5^{n\left( \frac{4}{5}E_0-N_0+1 \right) +\frac{1}{25}\left( 6^n-1 \right) E_0}\cdot 6^{n\left( N_0-\frac{4}{5}E_0-1 \right) +\frac{4}{25}\left( 6^n-1 \right) E_0}\cdot \tau \left( G \right) .
        \end{aligned}
    $$
\end{theorem} 

\begin{proof}\rm
    Denote the normalized Laplacian spectrum of $H_{n}(G)$ by $0 = \lambda^{(n)}_1 < \lambda^{(n)}_2 \leqslant \cdots \leqslant \lambda^{(n)}_{N_{n}}\leqslant 2$.  According to  \cite{Spectralgraphtheory},  $
        \tau \left( H_{n}(G) \right) =\frac{1}{2E_n}\prod_{i=1}^{N_n}{d^{(n)}_i}\prod_{i=2}^{N_n}{\lambda^{(n)}_i}
    $, where $d^{(n)}_i$ is the degree of vertex $i$ in $H_{n}(G)$.
    Since $  f_1\left( 2 \right) =\frac{3+\sqrt{5}}{2},f_2\left( 2 \right) =\frac{3-\sqrt{5}}{2},f_3\left( 2 \right) =2$, Whether $G$ is bipartite, we have
    \begin{equation}\label{eq32}
        \frac{\tau \left( H_n\left( G \right) \right)}{\tau \left( H_{n-1}\left( G \right) \right)}=\frac{2^{N_n}}{6}\cdot \frac{\prod_{i=2}^{N_n}{\lambda _{i}^{\left( n \right)}}}{\prod_{i=2}^{N_{n-1}}{\lambda _{i}^{\left( n-1 \right)}}}.
    \end{equation}
    And we can get by Theorem \ref{theo1},
    \begin{equation}\label{eq33}
        \begin{aligned}
            \prod_{i=2}^{N_n}{\lambda _{i}^{\left( n \right)}}
             & =\left( \frac{3}{4} \right) ^{N_{n-1}}\cdot \left( \frac{5}{4} \right) ^{E_{n-1}-N_{n-1}+1}\cdot \left( \frac{1}{4} \right) ^{E_{n-1}-N_{n-1}}\cdot \prod_{i=2}^{N_{n-1}}{\frac{\lambda _{i}^{\left( n-1 \right)}}{4}} \\
             & =\frac{3^{N_{n-1}}\cdot 5^{E_{n-1}-N_{n-1}+1}}{2^{4E_{n-1}}}\cdot \prod_{i=2}^{N_{n-1}}{\lambda _{i}^{\left( n-1 \right)}}.
        \end{aligned}
    \end{equation}
    Therefore, we get Eq.(\ref{eq30}) by combining Eqs.(\ref{eq32}) and (\ref{eq33}).

    Then we can also get
    $$
        \begin{aligned}
            \tau \left( H_n\left( G \right) \right) & =5^{\sum_{i=0}^{n-1}{\left[ \left( \frac{1}{5}\cdot 6^i+\frac{4}{5} \right) E_0-N_0+1 \right]}}\cdot 6^{\sum_{i=0}^{n-1}{\left[ N_0+\frac{4}{5}\left( 6^i-1 \right) E_0-1 \right]}}\cdot \tau \left( G \right) \\
                                                    & =5^{n\left( \frac{4}{5}E_0-N_0+1 \right) +\frac{1}{25}\left( 6^n-1 \right) E_0}\cdot 6^{n\left( N_0-\frac{4}{5}E_0-1 \right) +\frac{4}{25}\left( 6^n-1 \right) E_0}\cdot \tau \left( G \right) .
        \end{aligned}
    $$
\end{proof}

\subsection{The Kemeny's constant,  multiplicative degree-Kirchhoff index and spanning trees of \texorpdfstring{$H^k_{n}(G)$}{}}
Let $f_1(\sigma), f_2(\sigma), f_3(\sigma), f_4(\sigma)$, and $f_5(\sigma)$ be the three roots of Eq.(\ref{eq19}). By Vieta's Theorem, we obtain
\begin{equation}\label{eq34}
    \sum_{i=1}^5{\frac{1}{f_i\left( \sigma \right)}}=\frac{25k+5}{k+5}\cdot \frac{1}{\sigma}+\frac{40}{k+5},\quad \prod_{i=1}^5{f_i\left( \sigma \right)}=\frac{k+5}{16k+16}\cdot \sigma .
\end{equation}

Following are five multisets defined by a multiset $M$,
$$
    f_i\left( M \right) =\bigcup_{\chi \in M}{\left\{ f_i\left( \chi \right) \right\}},\ i=1,2,3,4,5.
$$

\begin{theorem} \label{theo45}
    The Kemeny's constant $K\left( H_{n}^{k}\left( G \right) \right)$ of $H^k_n(G)$ can be obtained by iteration of the following formula
    $$
        \begin{aligned}
            K\left( H_{n}^{k}\left( G \right) \right) & =\frac{25k+5}{k+5}K\left( H_{n-1}^{k}\left( G \right) \right) +\left[ \frac{8k\left( 5k+21 \right)}{5\left( k+5 \right)}\cdot \left( 5k+1 \right) ^{n-1}+\frac{32k}{5\left( k+5 \right)} \right] E_{0}^{k} \\
                                                      & \quad-\frac{8k}{k+5}N_{0}^{k}+\frac{4\left( 5k^2-23k \right)}{\left( k+5 \right) \left( 5k+1 \right)}.
        \end{aligned}
    $$
    In general,
    $$
        \begin{aligned}
            K\left( H_{n}^{k}\left( G \right) \right) & =\left( \frac{25k+5}{k+5} \right) ^nK\left( G \right) +\left[ \frac{8k\left( 5k+21 \right) \eta}{5\left( k+5 \right)}\cdot \left( 5k+1 \right) ^{n-1}+\frac{32k\mu}{5\left( k+5 \right)} \right] E_{0}^{k} \\
                                                      & \quad-\frac{8k\mu}{k+5}N_{0}^{k}+\frac{4\left( 5k^2-23k \right) \mu}{\left( k+5 \right) \left( 5k+1 \right)},
        \end{aligned}
    $$
    %where $\mu$ and $\eta$ are defined by Eq.(\ref{eq47}).
    where
    \begin{equation}\label{eq47}
        \begin{aligned}
            \mu =\frac{k+5}{24k}\left\{ \left[ \frac{5\left( 5k+1 \right)}{k+5} \right] ^n-1 \right\} ,
            \eta =-\frac{k+5}{k}\left\{ \left[ \frac{5}{k+5} \right] ^n-1 \right\}.
        \end{aligned}
    \end{equation}
\end{theorem} 
\begin{proof}\rm
    Similar to the proof of Theorem \ref{theo43}, we have
    $$
        \begin{aligned}
            K\left( H_{n}^{k}\left( G \right) \right) & =\sum_{i=2}^{N_{n-1}^{k}}{\left( \frac{25k+5}{k+5}\cdot \frac{1}{\lambda^{(n-1)} _i}+\frac{40}{k+5} \right)}+\frac{2\left( 5k+3 \right)}{5k+1}+8kE_{n-1}^{k}
            -8N_{n-1}^{k}+2.
        \end{aligned}
    $$
    % Similar to the proof of Eq.(\ref{eq29}), we can get
    % $$
    %     \begin{aligned}
    %         K\left( H_{n}^{k}\left( G \right) \right) 
    %         &=\left( \frac{25k+5}{k+5} \right) ^nK\left( G \right) +\left[ \frac{8k\left( 5k+21 \right) \eta}{5\left( k+5 \right)}\cdot \left( 5k+1 \right) ^{n-1}+\frac{32k\mu}{5\left( k+5 \right)} \right] E_{0}^{k}\\
    %         &\quad-\frac{8k\mu}{k+5}N_{0}^{k}+\frac{4\left( 5k^2-23k \right) \mu}{\left( k+5 \right) \left( 5k+1 \right)}.
    %     \end{aligned}    
    % $$
    % where $\mu$ and $\eta$ are defined by Eq.(\ref{eq47}).

    % This completes the proof.
    The results can be obtained by simplifying the above equation. 
    %This result is a direct consequence of Eq.(\ref{eq28}) and Theorem \ref{theo46}. 
\end{proof}

\begin{theorem} \label{theo46}
    The relationship between the multiplicative degree-Kirchhoff indices of $H^k_n(G)$ and $H^k_{n-1}(G)$ is
    $$
        \begin{aligned}
            Kf^{'}\left( H_{n}^{k}(G) \right) & =\frac{5\left( 5k+1 \right) ^2}{k+5}Kf^{'}\left( H_{n-1}^{k}(G) \right) -\frac{16k\left( 5k+1 \right) ^n}{k+5}E_{0}^{k}N_{0}^{k}+\left[\frac{64}{5\left( k+5 \right)}\cdot \left( 5k+1 \right) ^n
            \right.                                                                                                                                                                                                                               \\
                                              & \left.\quad
                +\frac{16k\left( 5k+21 \right)}{5\left( k+5 \right)}\cdot \left( 5k+1 \right) ^{2n-1} \right] \left( E_{0}^{k} \right) ^2+\frac{8\left( 5k^2-23k \right)}{\left( k+5 \right)}\cdot \left( 5k+1 \right) ^{n-1}E_{0}^{k}.
        \end{aligned}
    $$
    Thus, the general expression for  $Kf^{'}\left( H^k_n\left( G \right) \right)$ is
    \begin{equation}\label{eq46}
        \begin{aligned}
            Kf^{'}\left( H_{n}^{k}\left( G \right) \right) & =\left[ \frac{5\left( 5k+1 \right) ^2}{k+5} \right] ^nKf^{'}\left( G \right) -\frac{16k\left( 5k+1 \right) ^n\mu}{k+5}E_{0}^{k}N_{0}^{k}+\left[ \frac{64k\mu}{5\left( k+5 \right)}\cdot \left( 5k+1 \right) ^n
            \right.                                                                                                                                                                                                                                                         \\
                                                           & \left.\quad
                +\frac{16k\left( 5k+21 \right) \eta}{5\left( k+5 \right)}\cdot \left( 5k+1 \right) ^{2n-1} \right] \left( E_{0}^{k} \right) ^2+\frac{8\left( 5k^2-23k \right) \mu}{\left( k+5 \right)}\cdot \left( 5k+1 \right) ^{n-1}E_{0}^{k},
        \end{aligned}
    \end{equation}
    where $\mu$ and $\eta$ are defined by Eq.(\ref{eq47}).
\end{theorem} 

\begin{proof}\rm
    This result can be obtained from Theorem \ref{theo45} since $Kf^{'}\left( H^k_n(G) \right) =2E_n\cdot K\left( H^k_n(G) \right)$. 
\end{proof}

\begin{theorem} \label{theo48}
    The number of spanning trees of $H^k_n\left( G \right)$ is
    \begin{equation}
        \begin{aligned}\label{eq49}
            \tau \left( H_{n}^{k}\left( G \right) \right) =\left( k+5 \right) ^{N_{0}^{k}+\frac{4}{5}\left( \left( 5k+1 \right) ^{n-1}-1 \right) E_{0}^{k}-1}\cdot 5^{\left[ \frac{5k-4}{5}\left( 5k+1 \right) ^{n-1}+\frac{4}{5} \right] E_{0}^{k}-N_{0}^{k}+1}\cdot \tau \left( H_{n-1}^{k}\left( G \right) \right).
        \end{aligned}
    \end{equation}
    The general expression is
    $$
        \begin{aligned}
            \tau \left( H_{n}^{k}\left( G \right) \right) =\left( k+5 \right) ^{n\left( N_{0}^{k}-\frac{4}{5}E_{0}^{k}-1 \right) +\frac{4}{5}\xi E_{0}^{k}}\cdot 5^{n\left( \frac{4}{5}E_{0}^{k}-N_{0}^{k}+1 \right) +\frac{5k-4}{5}\xi E_{0}^{k}}\cdot \tau \left( G \right),
        \end{aligned}
    $$
    where
    \begin{equation}\label{eq48}
        \xi =\sum_{i=0}^{n-1}{\left( 5k+1 \right) ^i}=\frac{\left( 5k+1 \right) ^n-1}{5k}.
    \end{equation}
\end{theorem} 

\begin{proof}\rm
    % Let the normalized Laplacian eigenvalues of $H^k_n(G)$ be $0<\lambda _{1}^{\left( n \right)}\leqslant \lambda _{2}^{\left( n \right)}\leqslant \cdots \leqslant \lambda _{N^k_n}^{\left( n \right)}$. Since
    % $$
    % f_{1,2}\left( 2 \right) =\frac{3\pm \sqrt{5}}{2},f_{3,4}\left( 2 \right) \,\,=\frac{3k+5\pm \sqrt{5k^2+6k+5}}{4\left( k+1 \right)},f_5\left( 2 \right) =2,
    % $$  
    % Whether or not $G$ is bipartite, by Lemma\ref{lemma24}\ref{lem24_3} and the definition of $H^k_n(G)$  we have
    Similar to the proof of Theorem \ref{404}, we have
    \begin{equation}\label{eq50}
        \frac{\tau \left( H_{n}^{k}\left( G \right) \right)}{\tau \left( H_{n-1}^{k}\left( G \right) \right)}=\frac{\left( k+1 \right) ^{N_{n-1}^{k}}\cdot 2^{4kE_{n-1}^{k}}}{5k+1}\cdot \frac{\prod_{i=2}^{N^k_n}{\lambda _{i}^{\left( n \right)}}}{\prod_{i=2}^{N^k_{n-1}}{\lambda _{i}^{\left( n-1 \right)}}}.
    \end{equation}
    And we can get by Theorem \ref{theo2},
    \begin{equation}\label{eq51}
        \begin{aligned}
            \prod_{i=2}^{N^k_n}{\lambda _{i}^{\left( n \right)}}
            % &=\frac{5k+1}{4\left( k+1 \right)}\cdot \left( \frac{5}{4} \right) ^{kE_{n-1}^{k}-N_{n-1}^{k}+1}\cdot \left( \frac{1}{4} \right) ^{kE_{n-1}^{k}-N_{n-1}^{k}}\cdot \prod_{i=2}^{N^k_{n-1}}{\frac{k+5}{16k+16}\lambda _{i}^{\left( n-1 \right)}}\\
             & =\frac{\left( 5k+1 \right) \cdot \left( k+5 \right) ^{N_{n-1}^{k}-1}\cdot 5^{kE_{n-1}^{k}-N_{n-1}^{k}+1}}{\left( k+1 \right) ^{N_{n-1}^{k}}\cdot 2^{4kE_{n-1}^{k}}}\cdot \prod_{i=2}^{N_{n-1}^{k}}{\lambda _{i}^{\left( n-1 \right)}}.
        \end{aligned}
    \end{equation}
    Therefore, combining Eqs.(\ref{eq50}) and (\ref{eq51}), we can get Eq.(\ref{eq49}).

    Then we can also get
    $$
        \begin{aligned}
            \tau \left( H_{n}^{k}\left( G \right) \right) & =\left( k+5 \right) ^{\sum_{i=0}^{n-1}{\left\{ N_{0}^{k}+\frac{4}{5}\left( \left( 5k+1 \right) ^i-1 \right) E_{0}^{k}-1 \right\}}}\cdot 5^{\sum_{i=0}^{n-1}{\left\{ \left[ \frac{5k-4}{5}\left( 5k+1 \right) ^i+\frac{4}{5} \right] E_{0}^{k}-N_{0}^{k}+1 \right\}}}\cdot \tau \left( G \right) \\
            % &
            % =\left( k+5 \right) ^{n\left( N_{0}^{k}-\frac{4}{5}E_{0}^{k}-1 \right) +\frac{4}{5}\xi E_{0}^{k}}\cdot 5^{n\left( \frac{4}{5}E_{0}^{k}-N_{0}^{k}+1 \right) +\frac{5k-4}{5}\xi E_{0}^{k}}\cdot \tau \left( G \right) ,
        \end{aligned}
    $$
    The results can be obtained by simplifying the above equation. 
\end{proof}

\subsection{A new discovery and its explanation}
Interestingly, we find that although the Laplacian spectrum of $H_{n}(G)$  and $H^k_{n}(G)$  have different structures, but the formulas to calculate $K(H^k_{n}(G))$, $Kf^{'}(H^k_{n}(G))$ and $\tau \left( H_{n}^{k}\left( G \right) \right)$ (Theorems \ref{theo45}-\ref{theo48}) also hold when $k = 1$. We now give the reasons for this phenomenon.

Denote the roots of Eq.(\ref{eq18}) by $\lambda_1,\lambda_2,\lambda_3$. When $k=1$, Eq.(\ref{eq19}) can be written as
\begin{equation}\label{eq60}
    8\left( \lambda -\frac{1}{2} \right) \left( \lambda -\frac{3}{2} \right) \left[ 4\lambda ^3-\left( 10+2\sigma \right) \lambda ^2+\left( 5+4\sigma \right) \lambda -\sigma \right] =0,
\end{equation}
whose roots are $\lambda_1,\lambda_2,\lambda_3,\frac{1}{2},\frac{3}{2}$.

Because no matter what the value of $\sigma$  is, $\frac{1}{2}$ and $\frac{3}{2}$ are always the roots of Eq.(\ref{eq60}), combined with Theorem \ref{theo2}, we know that the multiplicity of $\frac{1}{2}$ and $\frac{3}{2}$ is $N_{n-1}$, so the normalized Laplacian spectrum of $H_n(G)$ and $H^k_n(G),k=1$ are the same from Theorem \ref{theo1} and Theorem \ref{theo2}.

%Therefore, Theorems \ref{theo45}-\ref{theo48} hold for all integers $k \geqslant 1$.

\section{Numerical Experiment}
Denote the $n$-th hexagonal graph  and k-hexagonal graph of a cycle of length 6 by $H_n$ and $H^k_n$. For $0\leqslant n\leqslant 8$, we list the Kemeny's constant of $H_n$ and $H^2_n$ in Table \ref{table2}, and multiplicative degree-Kirchhoff indices of $H_n$ and $H^2_n$ in Table \ref{table1}. Then we list The numbers of spanning trees of $H_n$ and $H^2_n$ for $0\leqslant n\leqslant 2$ in Table  \ref{table3}. It can be seen from Table  \ref{table3} that the numbers of spanning trees grows very fast. For $H^2_n$, it is difficult to calculate when $n > 2$.

\begin{table}[htb]
    \begin{center}
        \caption{The Kemeny's constant of $H_0,\cdots,H_8$ and $H^2_0,\cdots,H^2_8$}
        \label{table2}
        \begin{tabular}{cccccc}
            G       & $K(G)$      & G       & $K(G)$       & G       & $K(G)$        \\
            \hline
            $H_0$   & 0.89        & $H_1$   & 42.44        & $H_2$   & 458.22        \\
            $H_3$   & 3785.11     & $H_4$   & 27907.56     & $H_5$   & 193447.78     \\
            $H_6$   & 1290716.89  & $H_7$   & 8394470.44   & $H_8$   & 53617686.22   \\
            $H^2_0$ & 0.89        & $H^2_1$ & 87.92        & $H^2_2$ & 1622.01       \\
            $H^2_3$ & 23028.76    & $H^2_4$ & 294109.21    & $H^2_5$ & 3555757.33    \\
            $H^2_6$ & 41632025.66 & $H^2_7$ & 477742069.94 & $H^2_8$ & 5410653999.62
        \end{tabular}
    \end{center}
\end{table}

\begin{table}[htb]
    \begin{center}
        \caption{The multiplicative degree-Kirchhoff indices of $H_0,\cdots,H_8$ and $H^2_0,\cdots,H^2_8$}
        \label{table1}
        \begin{tabular}{cccccc}
            G       & $Kf^{'}(G) $    & G       & $Kf^{'}(G)$                 & G       & $Kf^{'}(G) $                 \\
            \hline
            $H_0$   & 10.67           & $H_1$   & 3056                        & $H_2$   & 197952                       \\
            $H_3$   & 9811008         & $H_4$   & 434018304                   & $H_5$   & 18050999040                  \\
            $H_6$   & 722636246016    & $H_7$   & 28198973740032              & $H_8$   & 1080685483941890             \\
            $H^2_0$ & 10.67           & $H^2_1$ & 11605.33                    & $H^2_2$ & 2355164.95                   \\
            $H^2_3$ & 367815398.31    & $H^2_4$ & 51672635965.05              & $H^2_5$ & 6871899281276.09             \\
            $H^2_6$ & 885044076026391 & $H^2_7$ & 11171809693029$\times 10^4$ & $H^2_8$ & 139178608420502$\times 10^5$
        \end{tabular}
    \end{center}
\end{table}

\begin{table}[htb]
    \begin{center}
        \caption{The numbers of spanning trees of $H_0,\cdots,H_2$ and $H^2_0,\cdots,H^2_2$}
        \label{table3}
        \begin{tabular}{cccccc}
            G       & $\tau(G)$ & G       & $\tau(G)$  & G       & $\tau(G)$                         \\
            \hline
            $H_0$   & 6         & $H_1$   & 241943     & $H_2$   & 6.71512031151729$\times 10^{32}$  \\
            $H^2_0$ & 6         & $H^2_1$ & 8426691368 & $H^2_2$ & 8.04003508846179$\times 10^{109}$
        \end{tabular}
    \end{center}
\end{table}

\section{Conclusion}
%  The normalized Laplacian spectra of iterated triangulations of a simple connected graph $G$ are determined in  \cite{tri}, whereas the normalized Laplacian spectra of the quadrilateral graph of $G$ were determined in  \cite{qua}. As well the degree-Kirchhoff indices, Kemeny's constants and number of spanning trees were studied.

% we consider the same problems on the $n$-th hexagonal graph $H_{n}(G)$ and $k$-hexagonal graph $H^k_{n}(G)$, which can be seen as an extension of the previous works mentioned as above. More specifically,
In this paper, we get the normalized Laplacian spectrum of $H_{n}(G)$ (Theorem \ref{theo1}) and $H^k_{n}(G)$ (Theorem \ref{theo2}) for any simple connected graph $G$. In addition, the formulas to calculate Kemeny's constants, the multiplicative degree-Kirchhoff indices, and the numbers of spanning trees of $H_{n}(G)$ and $H^k_{n}(G)$ are obtained. Finally, we give the reason for that
the formulas to the parameters  of $H^k_{n}(G)$, $k \geqslant 2$ are also hold when $k = 1$.

Different from the previous research on normalized Laplacian spectra \cite{tri,qua,k-tri,Pan,Chang}, We derive the normalized Laplacian spectrum of the $n$-th  hexagonal graph $H_{n}(G)$ and $k$-hexagonal graph $H^k_{n}(G)$ through the analysis of the structure of the solution of linear Equations (\ref{pro1_leqs}) and (\ref{pro2_leqs}). This analysis method is more convenient to be extended to $P^{k,t}(G)$, which obtained by substituting each edge of $G$ by a path of length 1 and $k$ paths of length $t$, $t > 5$.

\section*{Acknowledgments}
This work is supported by the Teaching Project of Nanjing Tech University (Project Number: 2019040), the Teaching Project of Jiangsu Higher Education Association (Project Number: 2020JDK023), and the National Natural Science Foundation of China (Grant/Award Number: 1771210).

\section*{Conflict of interest}
The authors declare that they have no known competing financial interests or personal relationships that could have appeared to influence the work reported in this paper.

\bibliographystyle{plain}
\bibliography{ref}

\end{document}